\documentclass[12pt,reqno]{amsart}
\usepackage[colorlinks=true, pdfstartview=FitV, linkcolor=blue, citecolor=blue, urlcolor=blue]{hyperref}
\usepackage{amssymb,amsmath, amscd}
\usepackage{ragged2e}
\usepackage{times, verbatim}
\usepackage{graphicx}
\usepackage[english]{babel}
 \usepackage[usenames, dvipsnames]{color}
\usepackage{amsmath,amssymb,amsfonts}
\usepackage{enumerate}
\usepackage{MnSymbol}
\usepackage{anysize}
\usepackage{enumitem}
\usepackage{bigints}
\usepackage{braket}
\usepackage[english]{babel}
\usepackage{blindtext}
\usepackage{mathtools}
\usepackage{graphicx}
\usepackage[a4paper,verbose]{geometry}
\marginsize{2.5cm}{2.5cm}{2.5cm}{2.5cm}
\input xy
\xyoption{all}
\usepackage{pb-diagram}
\usepackage[all]{xy}
\input xy
\xyoption{all}

\DeclareFontFamily{OT1}{rsfs}{}
\DeclareFontShape{OT1}{rsfs}{n}{it}{<-> rsfs10}{}
\DeclareMathAlphabet{\mathscr}{OT1}{rsfs}{n}{it}

\begin{document}
\theoremstyle{plain}

\newtheorem{theorem}{Theorem}[section]
\newtheorem{thm}[equation]{Theorem}
\newtheorem{prop}[equation]{Proposition}
\newtheorem{corollary}[equation]{Corollary}
\newtheorem{conj}[equation]{Conjecture}
\newtheorem{lemma}[equation]{Lemma}
\newtheorem{defn}[equation]{Definition}
\newtheorem{question}[equation]{Question}

\theoremstyle{definition}
\newtheorem{conjecture}[theorem]{Conjecture}

\newtheorem{example}[equation]{Example}
\numberwithin{equation}{section}

\newtheorem{remark}[equation]{Remark}

\newcommand{\Hecke}{\mathcal{H}}
\newcommand{\Liea}{\mathfrak{a}}
\newcommand{\Cmg}{C_{\mathrm{mg}}}
\newcommand{\Cinftyumg}{C^{\infty}_{\mathrm{umg}}}
\newcommand{\Cfd}{C_{\mathrm{fd}}}
\newcommand{\Cinftyfd}{C^{\infty}_{\mathrm{ufd}}}
\newcommand{\sspace}{\Gamma \backslash G}
\newcommand{\PP}{\mathcal{P}}
\newcommand{\bfP}{\mathbf{P}}
\newcommand{\bfQ}{\mathbf{Q}}
\newcommand{\Siegel}{\mathfrak{S}}
\newcommand{\g}{\mathfrak{g}}
\newcommand{\A}{\mathbb{A}}
\newcommand{\Q}{\mathbb{Q}}
\newcommand{\Gm}{\mathbb{G}_m}
\newcommand{\Nm}{\mathbb{N}m}
\newcommand{\ii}{\mathfrak{i}}
\newcommand{\II}{\mathfrak{I}}

\newcommand{\kk}{\mathfrak{k}}
\newcommand{\nn}{\mathfrak{n}}
\newcommand{\tF}{\widetilde{F}}
\newcommand{\p}{\mathfrak{p}}
\newcommand{\m}{\mathfrak{m}}
\newcommand{\bb}{\mathfrak{b}}
\newcommand{\Ad}{{\rm Ad}\,}
\newcommand{\ttt}{\mathfrak{t}}
\newcommand{\frakt}{\mathfrak{t}}
\newcommand{\U}{\mathcal{U}}
\newcommand{\Z}{\mathbb{Z}}
\newcommand{\bfG}{\mathbf{G}}
\newcommand{\bfT}{\mathbf{T}}
\newcommand{\R}{\mathbb{R}}
\newcommand{\ST}{\mathbb{S}}
\newcommand{\h}{\mathfrak{h}}
\newcommand{\bC}{\mathbb{C}}
\newcommand{\C}{\mathbb{C}}
\newcommand{\F}{\mathbb{F}}
\newcommand{\N}{\mathbb{N}}
\newcommand{\qH}{\mathbb {H}}
\newcommand{\temp}{{\rm temp}}
\newcommand{\Hom}{{\rm Hom}}
\newcommand{\Aut}{{\rm Aut}}
\newcommand{\rk}{{\rm rk}}
\newcommand{\Ext}{{\rm Ext}}
\newcommand{\End}{{\rm End}\,}
\newcommand{\Ind}{{\rm Ind}}
\newcommand{\ind}{{\rm ind}}
\newcommand{\Irr}{{\rm Irr}}
\def\circG{{\,^\circ G}}
\def\M{{\rm M}}
\def\diag{{\rm diag}}
\def\Ad{{\rm Ad}}
\def\As{{\rm As}}
\def\wG{{\widehat G}}
\def\G{{\rm G}}
\def\SL{{\rm SL}}
\def\PSL{{\rm PSL}}
\def\GSp{{\rm GSp}}
\def\PGSp{{\rm PGSp}}
\def\Sp{{\rm Sp}}
\def\St{{\rm St}}
\def\GU{{\rm GU}}
\def\SU{{\rm SU}}
\def\U{{\rm U}}
\def\GO{{\rm GO}}
\def\GL{{\rm GL}}
\def\PGL{{\rm PGL}}
\def\GSO{{\rm GSO}}
\def\GSpin{{\rm GSpin}}
\def\GSp{{\rm GSp}}

\def\Gal{{\rm Gal}}
\def\SO{{\rm SO}}
\def\O{{\rm  O}}
\def\Sym{{\rm Sym}}
\def\sym{{\rm sym}}
\def\St{{\rm St}}
\def\Sp{{\rm Sp}}
\def\tr{{\rm tr\,}}
\def\ad{{\rm ad\, }}
\def\Ad{{\rm Ad\, }}
\def\rank{{\rm rank\,}}

\def\Ext{{\rm Ext}}
\def\Hom{{\rm Hom}}
\def\Alg{{\rm Alg}}
\def\GL{{\rm GL}}
\def\SO{{\rm SO}}
\def\G{{\rm G}}
\def\U{{\rm U}}
\def\St{{\rm St}}
\def\Wh{{\rm Wh}}
\def\RS{{\rm RS}}
\def\ind{{\rm ind}}
\def\Ind{{\rm Ind}}
\def\csupp{{\rm csupp}}

\subjclass{Primary 11F70; Secondary 22E55}
\title{Archimedean Distinguished Representations and Exceptional Poles}
\author{Akash Yadav}
\address{Department of Mathematics \\
Indian Institute of Technology Bombay \\ Mumbai 400076}
\email{194090003@iitb.ac.in}
\begin{abstract}
Let $F$ be an archimedean local field and let $E$ be $F\times F$ (resp. a quadratic extension of $F$). We prove that an irreducible generic (resp. nearly tempered) representation of $\GL_n(E)$ is $\GL_n(F)$ distinguished if and only if its Rankin-Selberg (resp. Asai) $L$-function has an exceptional pole of level zero at $0$. Further,  we deduce a necessary condition for the ramification of such representations using the theory of weak test vectors developed by Humphries and Jo. 
\end{abstract}
\maketitle
\section{Introduction}

Let $F$ be an archimedean local field and let $E$ be either $F\times F$ (the split case) or a quadratic extension of $F$ (the inert case). By a smooth representation of $\GL_n(E)$, we mean a smooth admissible Fr\'echet representation of moderate growth in the sense of Casselman-Wallach. We say that an irreducible smooth representation $\pi$ of $\GL_n(E)$ is $\GL_n(F)$-distinguished, or simply distinguished, if $\operatorname{Hom}_{\GL_n(F)}(\pi,1)\neq 0$.

Baruch \cite{bar03} in the split case 
(resp. Kemarsky \cite{kem15a} in the inert case) established that for a distinguished representation of $\GL_n(E)$, linear forms invariant under the mirabolic subgroup $P_n(F)$ of $\GL_n(F)$ are, in fact, $\GL_n(F)$-invariant. This result allows us to relate the property of being distinguished to the exceptional poles of the Rankin-Selberg $L$-function in the split case (resp. the Asai $L$-function in the inert case). The notion of exceptional poles in the archimedean setting (\cite{jc15}) is somewhat more nuanced than in the $p$-adic case, as they can have various integer levels $m \geq 0$. Exceptional poles in the non-archimedean setting are essentially of level zero. 

Let $\psi: E\to \mathbb{S}^1$ be a fixed additive character trivial on $F$ and let $\pi$ be an irreducible generic representation of $\GL_n(E)$ having Whittaker model $\mathcal{W}(\pi,\psi)$. Let $\mathcal{S}=\mathcal{S}(F^n)$
denote the Schwartz space of $F^n$ and $e_n=(0,0,...,1)\in F^n$. We consider the family of Rankin-Selberg integrals \cite{jpss83} (resp. Flicker integrals \cite{f93}) 
$\mathcal{J}(\pi)=\{I(s, W, \Phi)\,\vert\, W\in\mathcal{W}(\pi,\psi),\Phi\in\mathcal{S}\}$, where
\[
I(s, W, \Phi)=\int_{N_n(F) \backslash \G L_n(F)} W(g)\Phi(e_n g)|\operatorname{det}(g)|_F^s \>d g.
\]
In \cite{jac09} (resp.\cite{bp21}), it is shown that as complex functions of $s$ these integrals converge in some right half-plane and are holomorphic multiples of the Rankin-Selberg (resp. Asai) $L$-function. 
We have a natural filtration $\mathcal{S}=\mathcal{S}^0 \supset \mathcal{S}^1 \supset \cdots \supset \mathcal{S}^m \supset \cdots$, with $\mathcal{S}^m$ defined as the set of functions $f \in \mathcal{S}$ vanishing to order at least $m$ at $0$.
At a pole $s_0$ of maximal order for the family $\mathcal{J}(\pi)$, we have, for each integral $I(s, W, \Phi)$, a Laurent expansion
\[
I(s, W, \Phi)=\frac{B_{s_0}\left(W, \Phi\right)}{\left(s-{s_0}\right)^d}+\cdots ,
\]
where $B_{s_0}:\mathcal{W}(\pi,\psi)\times\mathcal{S}\to \C$ is a bilinear form. We say that $s_0$ is an exceptional pole of level $m$ for $\mathcal{J}(\pi)$ if $B_{s_0}$ vanishes identically on $\mathcal{S}^{m+1}$ but not on $\mathcal{S}^m$. The main result of this paper is as follows.

\begin{thm}\label{1.1}
   Let $\pi$ be an irreducible generic representation of $\GL_n(E)$, assumed to be nearly tempered in the inert case. Then $\pi$ is distinguished if and only if the family of integrals $\mathcal{J}(\pi)$ has an exceptional pole of level zero at $0$.
\end{thm}
We refer to Section \ref{s2} for the definition of nearly tempered representations. 

In \cite{nm10}, Matringe established that for a quadratic extension $E/F$ of non-archimedean local fields, an irreducible generic representation of $\GL_n(E)$ is distinguished if and only if its Asai $L$-function has an exceptional pole at $0$. Our proof of Theorem $\ref{1.1}$ closely follows Matringe's approach. In the non-archimedean setting, the Asai $L$-function serves as the greatest common divisor of the family of Flicker integrals, ensuring that the quotient of these integrals and the Asai $L$-function is non-vanishing. However, in the archimedean case, the non-vanishing of this quotient is known only for nearly tempered representations. This compels us to restrict ourselves to the category of nearly tempered representations in the inert case.  

We briefly describe the content of each section of this paper. In Section \ref{s2}, we set up notation and conventions related to representations of the group $\GL_n$ and their $L$-functions. Next, we recall and prove some results on Rankin-Selberg and Flicker integrals in Section \ref{s3} and present the proof of Theorem \ref{1.1} in Sections \ref{s4} and \ref{s5}. Finally, in Section \ref{s6}, we establish some results on archimedean exceptional poles.

\section{Notation}\label{s2}

    Let $F$ be an archimedean local field and $E$ be either a quadratic extension of $F$ (the inert case) or $F\times F$ (the split case). We write $|\cdot|_F$ and $|\cdot|_E$ for the normalized absolute values of $F$ and $E$ respectively. Thus, in the split case we have $|(\lambda, \mu)|_E=$ $|\lambda|_F|\mu|_F$ for every $(\lambda, \mu) \in E$ and in both cases we have $|x|_E=|x|_F^2$ for every $x \in F$. Let $G_n$ denote the group $\GL_n$ ($n\in\mathbb{N}$) with $Z_n$, $B_n$ and $N_n$ being the subgroups of scalar, upper triangular and unipotent upper triangular matrices in $G_n$ respectively.  Let $P_n$ be the mirabolic subgroup of $G_n$ of matrices with last row $(0,\ldots,0,1)$. Let \(K_n\) represent the standard maximal compact subgroup of \(G_n(F)\), where \(K_n = O(n)\) if \(F = \mathbb{R}\), and \(K_n = U(n)\) if \(F = \mathbb{C}\).
 Let $\delta_n$ denote the modular character of $B_n(F)$. We denote the set of  isomorphism classes of irreducible admissible Casselman-Wallach representations of $G$ by $\operatorname{Irr}(G)$ and its subset of square integrable representations by $\Pi(G)$.

For $\pi\in\operatorname{Irr}\left(G_n(E)\right)$, we will denote by $\omega_\pi$ its central character. Let $P$ be a standard parabolic subgroup of $G_n$ and let $MU$ be its Levi decomposition. Then, $M$ is of the form
\[
M=G_{n_1} \times \ldots \times G_{n_k}
\]
for some integers $n_1, \ldots, n_k$ with $n_1+\ldots+n_k=n$. Let $\tau_i \in \operatorname{Irr}\left(G_{n_i}(E)\right)$, for every $1 \leqslant i \leqslant k$, be an irreducible representation of $G_{n_i}(E)$ so that $\sigma=\tau_1 \otimes \ldots \otimes \tau_k$ is an irreducible representation of $M(E)$. We will denote by
\[
\tau_1 \times \ldots \times \tau_k
\]
the normalized induced representation $i_{P(E)}^{G_n(E)}(\sigma)$. A representation $\pi\in\operatorname{Irr}\left(G_n(E)\right)$ is generic if it admits a nonzero Whittaker functional with respect to any (or equivalently one) generic character of $N_n(E)$. We will denote by $\operatorname{Irr}_{\operatorname{gen}}\left(G_n(E)\right)$ the subset of generic representations in $\operatorname{Irr}\left(G_n(E)\right)$. By Theorem 6.2f of \cite{dv78}, every $\pi\in \operatorname{Irr}_{\operatorname{gen}}(G_n(E))$ is isomorphic to a representation of
the form $\tau_1 \times \ldots \times \tau_k$, where for each $1\leq i\leq  k$, $\tau_i$ is an essentially square-integrable representation (i.e., an unramified twist of a square-integrable representation) of some $G_{n_i} (E)$. We will say that a representation $\pi$ of $G_n(E)$ is nearly tempered if it is isomorphic to an induced representation of the form:
\[
\left(\tau_1 \otimes|\operatorname{det}|_E^{\lambda_1}\right) \times \ldots \times\left(\tau_k \otimes|\operatorname{det}|_E^{\lambda_k}\right),
\]
where for each $1 \leqslant i \leqslant k, \tau_i \in \Pi\left(G_{n_j}(E)\right)$ for some $n_i \geqslant 1$ and $\lambda_i$ is a real number such that $|\lambda_i|<1/4$. By \cite{vs80}, every nearly tempered representation is irreducible and generic. We will denote by $\operatorname{Irr}_{\operatorname{ntemp}}(G_n(E))\subset\operatorname{Irr}_{\operatorname{gen}}(G_n(E))$ the subset of nearly tempered representations.

Fix non-trivial additive characters $\psi^{\prime}: F \rightarrow \mathbb{S}^1$ and $\psi: E \rightarrow \mathbb{S}^1$ and assume that $\psi$ is trivial on $F$. We define generic characters $\psi_n: N_n(E) \rightarrow \mathbb{S}^1, \psi_n^{\prime}: N_n(F) \rightarrow \mathbb{S}^1$ by
\[
\psi_n^{\prime}(u)=\psi^{\prime}\left((-1)^n \sum_{i=1}^{n-1} u_{i, i+1}\right) \text { and } \psi_n(u)=\psi\left((-1)^n \sum_{i=1}^{n-1} u_{i, i+1}\right) .
\]
If $\pi\in\operatorname{Irr}_{\operatorname{gen}}\left(G_n(E)\right)$, we
write $\mathcal{W}\left(\pi, \psi_n\right)$ for its Whittaker model (with
respect to $\psi_n$). Let $C^{\infty}
\left(N_n(E) \backslash G_n(E), \psi_n\right)$ denote the space of all smooth functions $W:G_n(E)\to\mathbb{C}$ such that $W(ug)=\psi_n(u)W(g)$ for every $u\in N_n(E),g\in G_n(E)$. For every Whittaker function $W \in C^{\infty}
\left(N_n(E) \backslash G_n(E), \psi_n\right)$, set $\widetilde{W} \in
C^{\infty}\left(N_n(E) \backslash G_n(E), \psi_n^{-1}\right)$ by

\[
\widetilde{W}(g)=W\left(w_n{ }^t g^{-1}\right) \quad\left(g \in G_n(E)\right),
\]
where
$$w_n= \begin{bmatrix}
 &  &  &  & 1 \\
 &  &  & \cdot &  \\
 &  & \cdot &  &  \\
 & \cdot &  &  &  \\
1 &  &  &  &  
\end{bmatrix}.$$ 
Then the map $W \mapsto \widetilde{W}$ induces a (topological) isomorphism $\mathcal{W}\left(\pi, \psi_n\right) \simeq$ $\mathcal{W}\left(\widetilde{\pi}, \psi_n^{-1}\right)$, where $\widetilde{\pi}$ denotes the smooth contragredient of $\pi$.

We let $\mathcal{S}\left(F^n\right)$ be the Schwartz space on $F^n$. We denote by $\Phi \mapsto \widehat{\Phi}$ be the usual Fourier transform on $F^n$ defined using the additive character $\psi^{\prime}$ and the corresponding self-dual measure: for every $\Phi \in \mathcal{S}\left(F^n\right)$ we have
$$
\widehat{\Phi}\left(x_1, \ldots, x_n\right)=\int_{F^n} \Phi\left(y_1, \ldots, y_n\right) \psi^{\prime}\left(x_1 y_1+\ldots+x_n y_n\right) d y_1 \ldots d y_n 
$$
for all $\left(x_1, \ldots, x_n\right) \in F^n$, where the measure of integration is chosen so that $\widehat{\widehat{\Phi}}(v)=\Phi(-v)$. Let $e_n=(0, \ldots, 0,1) \in F^n$.

Let $W_F$ be the Weil group of $F$. In the inert case, we define similarly the Weil group $W_E$ of $E$. An admissible complex representation of $W_F$ (resp. $W_E$ ) is by definition a continuous morphism $\phi: W_F \rightarrow \GL(V)$ (resp. $\phi: W_E\rightarrow \GL(V)$ ), where $V$ is a finite dimensional complex vector space. To any admissible complex representation $\phi: W_F\rightarrow \GL(V)$, we associate a local $L$-factor $L(s, \phi)$ and a local $\epsilon$-factor $\epsilon\left(s, \phi, \psi^{\prime}\right)$ as in \cite{ta79}. 

Let $\eta_{E / F}$ be the quadratic character of $W_F$ associated to the extension $E / F$ and set
$$
\lambda_{E / F}\left(\psi^{\prime}\right)=\epsilon\left(\frac{1}{2}, \eta_{E / F}, \psi^{\prime}\right) .
$$
This is called the Langlands constant of the extension $E/F$ . It is a fourth root of unity which is trivial in the split case. 

In the split case, we set $\tau$ to be the unique element $(\beta,-\beta)\in F^{\times}\times F^{\times}$ such that $$\psi(x,y)=\psi^{\prime}(\beta x)\psi^{\prime}(-\beta y).$$

For $\pi=\pi_1\otimes\pi_2\in\operatorname{Irr}_{\operatorname{gen}}\left(G_n(F\times F)\right)$ and $\tau=(1,-1)$, every $W\in\mathcal{W}\left(\pi, \psi_n\right)$ can be written as $W_1\otimes W_2$ for some $W_1\in\mathcal{W}\left(\pi_1, \psi^{\prime}_n\right)$ and $W_2\in\mathcal{W}\left(\pi_2, {\psi^{\prime}_n}^{-1}\right).$ Then for $\Phi\in\mathcal{S}$, the integral $I(s,W,\Phi)$ equals
\[I(s,W,\Phi)=\int_{N_n(F) \backslash \G L_n(F)} W_1(g)W_2(g)\Phi(e_n g)|\operatorname{det}(g)|_F^s \>d g.\]

In the inert case, we set $\tau$ to be the unique element in $E$ such that $\psi(z)=\psi'(\operatorname{Tr}_{E/F}(\tau z))$ for every $z\in E$, where $\operatorname{Tr}_{E/F}$ stands
for the trace of the extension $E/F$. 

Assume that we are in the split case. Let $\pi=\pi_1 \otimes \pi_2$ be an irreducible representation of $G_n(E)=G_n(F) \times G_n(F)$. Let $\phi_1, \phi_2: W_F\rightarrow \GL_n(\mathbb{C})$ be the admissible complex representations associated to $\pi_1$ and $\pi_2$ respectively by the local Langlands correspondence (\cite{rpl89}). Then, we define
\[
\begin{aligned}
& L\left(s, \pi,\mathrm{As}\right):=L\left(s, \phi_1 \otimes \phi_2\right), \\
&\epsilon\left(s, \pi, \psi^{\prime},\mathrm{As}\right):=\epsilon\left(s, \phi_1 \otimes \phi_2, \psi^{\prime}\right).
\end{aligned}
\]
We remark that $L\left(s, \pi,\mathrm{As}\right)$ is more commonly denoted by $L\left(s, \pi_1 \times \pi_2\right)$ in the literature, and is usually called the Rankin-Selberg convolution or the tensor product $L$-function of $\pi_1$ and $\pi_2$. Our notation ensures that our results can be stated uniformly for the split and inert cases.

Assume now that we are in the inert case, i.e., $E=\mathbb{C}$ and $F=\mathbb{R}$. Then $W_E=\mathbb{C}^*$ and $W_F=\mathbb{C}^*\bigcup j\mathbb{C}^*$, where $j^2=-1$ and $jwj^{-1}=\overline{w}$ for all $w\in\mathbb{C}^*$ (here bar denotes the complex conjugation). Let $\phi: W_E \rightarrow \GL(V)$ be an admissible representation. We define $\operatorname{As}(\phi): W_F\mapsto \GL(V \otimes V)$ by 
$$
\begin{aligned}
& \operatorname{As}(\phi)(w)=\phi(w) \otimes \phi\left(\overline{w}\right) \text{ for } w \in \mathbb{C}^* \\
&\operatorname{As}(\phi)(j)=\left(\operatorname{Id}_V \otimes \>\phi\left(-1\right)\right) \circ \iota,
\end{aligned}
$$   
where $\iota$ is the linear automorphism of $V \otimes V$ sending $u \otimes v$ to $v \otimes u$. Then, $\operatorname{As}(\phi)$ is an admissible representation of $W_F$ and we call $
L(s, \operatorname{As}(\phi))$ and $\epsilon\left(s, \operatorname{As}(\phi), \psi^{\prime}\right)$ the Asai $L$-function and $\epsilon$-factor of $\phi$ respectively. 

Let $\pi \in \operatorname{Irr}\left(G_n(E)\right)$. Then, the local Langlands correspondence for $\GL_n$ associates to $\pi$ an admissible complex representation $\phi_\pi: W_E \rightarrow$ $\GL(V)$ of dimension $n$. We set 
\[
\begin{aligned}
&L(s, \pi, \mathrm{As}):=L(s, \operatorname{As}(\phi))\\
&\epsilon(s, \pi, \mathrm{As}):=\epsilon(s, \operatorname{As}(\phi)),
\end{aligned}
\]
and call them the Asai $L$-function and $\epsilon$-factor of $\pi$ respectively. 

\section{Preliminary Results}\label{s3}
For $\pi\in\operatorname{Irr}_{\operatorname{gen}}\left(G_n(E)\right)$, consider as before the family of integrals
\[
\mathcal{J}(\pi)=\{I(s,W,\Phi):W\in\mathcal{W}(\pi,\psi_n),\Phi\in\mathcal{S}(F^n)\}.
\]
The following theorem combines the work of Jacquet-Piatetskii-Shapiro-Shalika in \cite{jpss83} and Jacquet in \cite{jac09} for the split case and of Beuzart-Plessis in \cite{bp21} for the inert case.

\begin{thm}\label{3.1}
For $\pi\in\operatorname{Irr}_{\operatorname{gen}}(G_n(E))$, let $W \in \mathcal{W}(\pi, \psi_n)$ and $\Phi \in \mathcal{S}(F^n).$ 

\begin{enumerate}

    \item The integrals $I(s, W, \Phi)$ are convergent when $\operatorname{Re}(s)\gg 0$, and moreover extend to meromorphic functions on $\mathbb{C}$. Further, if $\pi\in\operatorname{Irr}_{\operatorname{ntemp}}(G_n(E))$, these integrals converge absolutely when $\operatorname{Re}(s)\geq 1/2$. \\
    \item We have the functional equation
\[
\qquad\quad\frac{I(1-s, \widetilde{W}, \widehat{\Phi})}{L(1-s, \widetilde{\pi}, \mathrm{As})}= \omega_\pi(\tau)^{n-1}|\tau|_E^{\frac{n(n-1)}{2}(s-1 / 2)} \lambda_{E / F}\left(\psi^{\prime}\right)^{-\frac{n(n-1)}{2}}\epsilon\left(s, \pi, \mathrm{As}, \psi^{\prime}\right) \frac{I(s, W, \Phi)}{L(s, \pi, \mathrm{As})}.
\]

\item The function $s \mapsto \dfrac{I(s, W, \Phi)}{L(s, \pi, \mathrm{As})}$ is holomorphic. Further, in the split case, for any $s_0 \in \mathbb{C}$, one can choose $W \in \mathcal{W}(\pi, \psi_n)$ and $\Phi \in \mathcal{S}\left(F^n\right)$ such that this function does not vanish at $s_0$. Such a choice can be made in the inert case as well provided $\pi$ is nearly tempered.
\end{enumerate}

\end{thm}

We first recall the following proposition due to Kemarsky \cite{kem15a} in the inert case and Baruch \cite{bar03} in the split case.

\begin{prop}\label{3.2}
    For a distinguished representation $\pi\in\operatorname{Irr}_{\operatorname{gen}}\left(G_n(E)\right)$, any $P_n(F)$-invariant linear form on the space of $\pi$ is actually $G_n(F)$-invariant.
\end{prop}

Suppose $\pi\in\operatorname{Irr}_{\operatorname{gen}}\left(G_n(E)\right)$ and $s=s_0$ is a pole of maximal order $d$ for the family of integrals $\mathcal{J}(\pi)$, then we have an expansion
\[
I(s,W,\Phi)=\frac{B_{s_0}(W,\Phi)}{(s-s_0)^d}+\cdots ,
\]
where $B_{s_0}$ is a bilinear form on $\mathcal{W}(\pi,\psi_n)\times\mathcal{S}(F^n)$ satisfying the invariance property
$$ B_{s_0}(g.W,g.\Phi) = \left|\operatorname{det} g\right|^{{-s_0}}B_{s_0}(W,\Phi)$$
for every $g\in G_n(E), W\in\mathcal{W}(\pi,\psi_n),\Phi\in\mathcal{S}(F^n).$

Note that $\mathcal{S}^1$ is a codimension $1$ subspace of $\mathcal{S}$, being the kernel of the evaluation at zero map from $\mathcal{S}\to\mathbb{C}$. Therefore, if $s=s_0$ is an exceptional pole of level $0$ for the family $\mathcal{J}(\pi)$, then $B_{s_0}$ is of the form $B_{s_0}(W,\Phi)=\lambda_{s_0}(W)\Phi(0)$, where $\lambda_{s_0}$ is a non-zero linear form on $\mathcal{W}(\pi,\psi_n)$ such that $\lambda_{s_0}(g.W)=\vert\det(g)\vert^{-s_0}\lambda_{s_0}(W)$ for all $g\in G_n(E),\>W\in\mathcal{W}(\pi,\psi_n)$. Hence, we have:

\begin{prop}\label{3.3}
    Let $\pi\in\operatorname{Irr}_{\operatorname{gen}}\left(G_n(E)\right)$ be such that $s=0$ is an exceptional pole of level $0$ for the family $\mathcal{J}(\pi)$. Then $\pi$ is distinguished.
\end{prop}
For $W\in\mathcal{W}(\pi,\psi_n)$, consider the integral
\begin{align*}
  \Psi(s,W)&=\int_{N_n(F) \backslash P_n(F)} W(p)|\operatorname{det}(p)|_F^{s-1} \>d p\\ &=\int_{N_{n-1}(F) \backslash G_{n-1}(F)} W\begin{pmatrix}
g & 0 \\
0 & 1 
\end{pmatrix}|\operatorname{det}(g)|_F^{s-1} \>d g.  \end{align*}

\begin{prop}\label{3.4}
    The integral $\Psi(s,W)$ belongs to the complex vector space spanned by the elements of the family $\mathcal{J}(\pi)$ for all $W\in\mathcal{W}(\pi,\psi_n)$.
\end{prop}

\begin{proof}
The proof of the proposition is essentially identical to that of Lemma $14.1$ in \cite{jac09} so we omit it.
\end{proof}

\begin{prop}\label{3.5}
Let $\pi\in\operatorname{Irr}_{\operatorname{ntemp}}\left(G_n(E)\right).$ Then the linear form $l$ on $\mathcal{W}(\pi,\psi_n)$ defined by
\[
l(W)=\Psi(1,W)=\int_{N_n(F) \backslash P_n(F)} W(p)\>d p
\]
is non-trivial. Further, if $\pi$ is distinguished, then $l$ is $G_n(F)$- invariant. 
\end{prop}

\begin{proof}
 As $\pi$ is nearly tempered, the integrals $\Psi(s,W)$ converge at $s=1$ by Theorem \ref{3.1}($1$). Choose $f\in C_c^{\infty}(N_n(E)\backslash P_n(E),\psi_n)$ such that $$ \int_{N_n(F) \backslash P_n(F)} f(p)\>d p\neq 0. $$
Then by \cite{hj10} (in the split case) and \cite{kem15b} (in the inert case), there exists $W_0\in\mathcal{W}(\pi,\psi_n)$ such that its restriction to $P_n(E)$ agrees with $f$. Hence, $l(W_0)\neq 0$. Finally, if $\pi$ is distinguished, then  as $l$ is $P_n(F)$-invariant, it must be $G_n(F)$-invariant by Proposition \ref{3.2}.
\end{proof}

\begin{remark}\label{3.6}
    Let $\pi\in\operatorname{Irr}_{\operatorname{ntemp}}\left(G_n(E)\right)$ be distinguished.  Then so is $\widetilde{\pi}$. Thus, by Proposition \ref{3.5}, we have a non-trivial $G_n(F)$-invariant linear form $l'$ on $\mathcal{W}(\widetilde{\pi},\psi_n^{-1})$ defined as 
\[
l'(\widetilde{W})=\int_{N_n(F) \backslash P_n(F)} \widetilde{W}(p)\>dp.
\]
\end{remark}
Next, we recall the following continuity result in the split case.

\begin{prop}\label{3.7}
    Let $\pi\in\operatorname{Irr}_{\operatorname{gen}}\left(G_n(F\times F)\right)$. For every $s\in\mathbb{C}$ and $\Phi\in\mathcal{S}$, the linear forms on $\mathcal{W}(\pi,\psi_n)$ defined by
    \[
    \Lambda_{s,\Phi}(W)=\frac{I(s,W,\Phi)}{L(s,\pi,\mathrm{As})}\quad\text{and}\quad
    \Lambda^{\prime}_{s}(W)=\frac{\Psi(s,W)}{L(s,\pi,\mathrm{As})}
    \]
    are continuous.
\end{prop}

\begin{proof}
    The continuity of $\Lambda_{s,\Phi}$ follows from Theorem $1.1$ of \cite{cp04}, and that of $\Lambda^{\prime}_{s}$ follows from the proof of Proposition $4.2$ of \cite{jc15}.
\end{proof}

We end this section with the following non-vanishing result which emerges directly from Lemma $3.3.3$ of \cite{bp21}. 

\begin{prop}\label{3.8}
Let $\pi\in\operatorname{Irr}_{\operatorname{gen}}(G_n(E))$. Then for every $s_0\in\mathbb{C}$, there exists $W_1 \in \mathcal{W}(\pi, \psi_n)$ and $\Phi_1 \in \mathcal{S}^1$ such that the function $s\mapsto I(s,W_1,\Phi_1)$ does not vanish at $s=s_0$.   
\end{prop}

\section{Proof of Theorem 1.1 : Split Case}\label{s4}

Let $\pi=\pi_1\otimes\pi_2\in\operatorname{Irr}_{\operatorname{gen}}\left(G_n(F\times F)\right)$. We remind the reader that the Rankin-Selberg $L$-function traditionally denoted by $L(s,\pi_1\times\pi_2)$ is denoted by $L(s,\pi,\mathrm{As})$ in this paper, so that the split and inert cases can be treated uniformly.

We only need to prove that if $\pi$ is distinguished, then $s=0$ is an exceptional pole of level $0$ for the family $\mathcal{J}(\pi).$ For $\operatorname{Re}(s)<<0$, consider the integral
\[
I(1-s, \widetilde{W}, \widehat{\Phi})=\int_{N_n(F) \backslash G_n(F)} \widetilde{W}(g) \widehat{\Phi}(e_n g)|\operatorname{det} g|_F^{1-s} \>d g,
\]
where $W \in \mathcal{W}(\pi, \psi_n)$ and $\Phi \in \mathcal{S}\left(F^n\right)$.
We assume that if $W$ has the tensor product decomposition $W_1\otimes W_2$ for some Whittaker functions $W_1$ and $W_2$ of representations $\pi_1$ and $\pi_2$ respectively, then both $W_1$ and $W_2$ are $K_n$-finite. Therefore, there exist finitely many Whittaker functions $W_l'$ (for $\pi_1$), $W_j''$ (for $\pi_2$) and continuous functions $h_l,g_j$ on $K_n$ such that
\[\widetilde{W_1}(gk)=\sum\limits_{l}h_l(k)\widetilde{W}_l'(g),\>\>\widetilde{W_2}(gk)=\sum\limits_{j}g_j(k)\widetilde{W}_j''(g)\]
for all $k\in K_n$, and $g\in G_n(F).$ Set $f_i=h_lg_j$ and $W_i=W_l'\otimes W_j''$, where $i$ varies as the indices $j$ and $l$ vary. Then 
\[\widetilde{W}(gk)=\sum\limits_{i}f_i(k)\widetilde{W}_i(g).\]
As $\pi$ is distinguished, its central character $\omega_{\pi}$ is trivial on the diagonal of $F^{\times}\times F^{\times}$.
Using the Iwasawa decomposition, we get
\[
I(1-s, \widetilde{W}, \widehat{\Phi})=\int\limits_{K_n}\int\limits_{N_n(F)\backslash P_n(F)}\widetilde{W}(pk)\left|\operatorname{det}(p)\right|^{-s}\>dp\int\limits_{{F}^*}\widehat{\Phi}(e_nak)|a|^{n(1-s)}d^{\times}a dk.
\]
Dividing by the $L$-function, we get
\begin{align*}
   \frac{I(1-s, \widetilde{W}, \widehat{\Phi})}{L(1-s,\widetilde{\pi},\mathrm{As})}&=\int\limits_{K_n}\frac{\Psi(1-s,\>\widetilde{\pi}(k)\widetilde{W})}{L(1-s,\widetilde{\pi},\mathrm{As})}\int\limits_{{F}^*}\widehat{\Phi}(e_nak)|a|^{n(1-s)}d^{\times}a dk.\\
   &=\sum\limits_{i}\frac{\Psi(1-s,\>\widetilde{W}_i)}{L(1-s,\widetilde{\pi},\mathrm{As})}\int\limits_{K_n}\int\limits_{{F}^*}f_i(k)\widehat{\Phi}(e_nak)|a|^{n(1-s)}d^{\times}a dk.
\end{align*}
Notice that there is a positive real number $\epsilon$ such that for $\operatorname{Re}(s)<\epsilon$,
$$\int\limits_{K_n}\int\limits_{{F}^*}f_i(k)\widehat{\Phi}(e_nak)|a|^{n(1-s)}d^{\times}a dk$$
is absolutely convergent and defines a holomorphic function for each $i$.
So, in the previous equation, we have an equality of analytic functions for $\operatorname{Re}(s)<\epsilon$. Substituting $s=0$, we obtain
 $$\frac{I(1, \widetilde{W}, \widehat{\Phi})}{L(1,\widetilde{\pi},\mathrm{As})}=\int\limits_{K_n}\frac{\Psi(1,\>\widetilde{\pi}(k)\widetilde{W})}{L(1,\widetilde{\pi},\mathrm{As})}\int\limits_{{F}^*}\widehat{\Phi}(e_nak)|a|^{n}d^{\times}a dk.$$
As $\pi$ is distinguished, so is $\widetilde{\pi}$. Therefore, by Proposition \ref{3.2}, the $P_n(F)$-invariant linear form 
$$ \widetilde{W}\mapsto \frac{\Psi(1,\widetilde{W})}{L(1,\widetilde{\pi},\mathrm{As})}$$
on $\mathcal{W}(\widetilde{\pi}, \psi_n^{-1})$ must be $G_n(F)$-invariant. Thus, we obtain
\begin{align*}
     \frac{I(1, \widetilde{W}, \widehat{\Phi})}{L(1,\widetilde{\pi},\mathrm{As})}&=\frac{\Psi(1,\widetilde{W})}{L(1,\widetilde{\pi},\mathrm{As})}\int\limits_{K_n}\int\limits_{{F}^*}\widehat{\Phi}(e_nak)|a|^{n}d^{\times}a dk\\
     &=\frac{\Psi(1,\widetilde{W})}{L(1,\widetilde{\pi},\mathrm{As})}\Phi(0).
\end{align*}
As the space of $K_n\times K_n$-finite Whittaker functionals is dense in $\mathcal{W}({\pi}, \psi_n)$, by Proposition  \ref{3.7}, the equality
\[\frac{I(1, \widetilde{W}, \widehat{\Phi})}{L(1,\widetilde{\pi},\mathrm{As})}=\frac{\Psi(1,\widetilde{W})}{L(1,\widetilde{\pi},\mathrm{As})}\Phi(0)\]
must hold for all $W\in\mathcal{W}(\pi, \psi_n),\>\Phi\in\mathcal{S}$. 
Using Theorem \ref{3.1}(2) at $s=0$, we get
\[
\frac{I(1, \widetilde{W}, \widehat{\Phi})}{L(1, \widetilde{\pi}, \mathrm{As})}=\alpha\>\frac{I(0, W, \Phi)} {L(0, \pi, \mathrm{As})}
\]
for some $\alpha\in\mathbb{C}^{\times}$. From the earlier calculation, the left hand side of the above equation vanishes for any $\Phi\in\mathcal{S}^1$, and hence, ${I(0,W,\Phi)}/{L(0, \pi, \mathrm{As})}=0$ for all $W\in\mathcal{W}(\pi,\psi_n),\>\Phi\in\mathcal{S}^1$. Using Proposition \ref{3.8}, we choose $W_1$ and $\Phi_1$ in such a way that $\Phi_1\in\mathcal{S}^1$ and the function $s\mapsto I(s, W_1, \Phi_1)$ does not vanish at $s=0$. But as $I(0, W_1, \Phi_1)/L(0, \pi, \mathrm{As})=0$, it follows that $s=0$ is a pole for $L(s, \pi, \mathrm{As})$, say of order $d$. By Theorem \ref{3.1}(3), $s=0$ must be a pole for the family $\mathcal{J}(\pi)$ of maximal order $d$ . Observe that
\[B_{0}(W,\Phi)=\left(\lim_{s\to 0} s^d L(s,\pi,\mathrm{As})\right)\frac{I(0,W,\Phi)}{L(0, \pi, \mathrm{As})}\]
for all $W\in\mathcal{W}(\pi,\psi_n),\>\Phi\in\mathcal{S}$. By the above discussion, the bilinear form $B_0$ is non-trivial and vanishes on $\mathcal{S}^1$. Therefore, $s=0$ must be an exceptional pole of level $0$ for the family $\mathcal{J}(\pi)$.

\section{Proof of Theorem 1.1 : Inert Case}\label{s5}

We only need to prove that if $\pi$ is nearly tempered and distinguished, then $s=0$ is an exceptional pole of level $0$ for the family $\mathcal{J}(\pi).$
    
We make suitable choices of Haar measures $dg$, $dp$ and $\mu_2$ on $N_n(F) \backslash G_n(F)$, $N_n(F) \backslash P_n(F)$ and $P_n(F) \backslash G_n(F)$ respectively (see the proof of Theorem 1.4 in \cite{akt04} for details).

Consider
\[
I(1, \widetilde{W}, \widehat{\Phi})=\int_{N_n(F) \backslash G_n(F)} \widetilde{W}(g) \widehat{\Phi}(e_n g)|\operatorname{det} g|_F d g,
\]
where $W \in \mathcal{W}(\pi, \psi_n)$ and $\Phi \in \mathcal{S}\left(F^n\right)$. This is well defined as $\pi$ is nearly tempered. Following the argument in \cite{akt04},
\[
\begin{aligned}
I(1, \widetilde{W}, \widehat{\Phi}) & =\int_{P_n(F) \backslash G_n(F)} \int_{N_n(F) \backslash P_n(F)} \widetilde{W}(p g) \widehat{\Phi}(e_n p g) dp d\mu_2(g) \\
& =\int_{P_n(F) \backslash G_n(F)} \widehat{\Phi}(e_n g)\left(\int_{N_n(F) \backslash P_n(F)} \widetilde{W}(p g) dp\right) d\mu_2(g) \\
& =\int_{P_n(F) \backslash G_n(F)} \widehat{\Phi}(e_n g)\left(\int_{N_n(F) \backslash P_n(F)}(\widetilde{\pi}(g) \widetilde{W})(p) dp\right) d\mu_2(g) \\
& =\int_{P_n(F) \backslash G_n(F)} \widehat{\Phi}(e_n g) d\mu_2(g) \times\left(\int_{N_n(F) \backslash P_n(F)} \widetilde{W}(p) dp\right) \\
& =c \cdot \Phi(0) \cdot \ell^{\prime}(\widetilde{W}),
\end{aligned}
\]
where $c$ is a non-zero constant. Note that the fourth equality is obtained because we know that $l'$ gives a $G_n(F)$-invariant linear form on $\widetilde{\pi}$ by Remark \ref{3.6}. The integral over $P_n(F) \backslash G_n(F)$ equals $\Phi(0)$ by Fourier inversion, since $P_n(F) \backslash G_n(F)\cong F^n\setminus \{0\}$ and $d\mu_2$ is a Haar measure on $F^n$.

Using Theorem \ref{3.1}(2) at $s=0$, we get
\[
\frac{I(1, \widetilde{W}, \widehat{\Phi})}{L(1, \widetilde{\pi}, \mathrm{As})}=\alpha\>\frac{I(0, W, \Phi)} {L(0, \pi, \mathrm{As})}
\]
for some $\alpha\in\mathbb{C}^{\times}$. From the earlier calculation, the left hand side of the above equation vanishes for any $\Phi\in\mathcal{S}^1$, and hence, ${I(0,W,\Phi)}/{L(0, \pi, \mathrm{As})}=0$ for all $W\in\mathcal{W}(\pi,\psi_n),\>\Phi\in\mathcal{S}^1$. Using Proposition \ref{3.8}, we choose $W_1$ and $\Phi_1$ in such a way that $\Phi_1\in\mathcal{S}^1$ and the function $s\mapsto I(s, W_1, \Phi_1)$ does not vanish at $s=0$. But as $I(0, W_1, \Phi_1)/L(0, \pi, \mathrm{As})=0$, it follows that $s=0$ is a pole for $L(s, \pi, \mathrm{As})$, say of order $d$. As $\pi$ is nearly tempered, by Theorem \ref{3.1}(3), $s=0$ must be a pole for the family $\mathcal{J}(\pi)$ of maximal order $d$ . Observe that

\[B_{0}(W,\Phi)=\left(\lim_{s\to 0} s^d L(s,\pi,\mathrm{As})\right)\frac{I(0,W,\Phi)}{L(0, \pi, \mathrm{As})}\]
for all $W\in\mathcal{W}(\pi,\psi_n),\>\Phi\in\mathcal{S}$. By the above discussion, the bilinear form $B_0$ is non-trivial and vanishes on $\mathcal{S}^1$. Therefore, $s=0$ must be an exceptional pole of level $0$ for the family $\mathcal{J}(\pi)$.

\section{Results on Exceptional Poles}\label{s6}
    In this section, we will prove some results on archimedean exceptional poles.
We begin by recalling the notion of archimedean conductor exponent (\cite{ph20}). 

    Let $\pi$ be an irreducible generic representation of $G_n(F)$ with representation space $V_{\pi}$. Let $\widehat{K_n}$ denote the set of equivalence class of irreducible representations of $K_n$. We call $\tau \in \widehat{K_n}$ a $K_n$-type of $\pi$ if $\operatorname{Hom}(\tau,\pi|_{K_n})$ is non-trivial. Let $V_{\pi}^{\tau}$ denote the $\tau$-isotypic subspace of $V_{\pi}$. Then set
    \[V_{\pi}^{\tau|_{K_{n-1}}}\coloneqq\left\{v \in V_\pi^\tau: \pi\begin{pmatrix}
k^{\prime} & 0 \\
0 & 1
\end{pmatrix}\cdot v=v \text { for all } k^{\prime} \in K_{n-1}\right\} .\]
Let $\operatorname{deg}(\tau)$ denote the Howe degree of $\tau\in\widehat{K_n}$ (see  \cite{rh89} or $4.2$ of \cite{ph20} for the definition). It is a non-negative integer.

\begin{defn}
   The archimedean conductor exponent is defined as
   $$ c(\pi)\coloneqq \operatorname{min}\{\operatorname{deg}(\tau): \tau\in \widehat{K_n} \text{ is a } K_n \text{-type and } V_{\pi}^{\tau|_{K_{n-1}}}\neq\{0\}\}.$$   
\end{defn}

\begin{remark}
 For a representation $\pi=\pi_1\otimes\pi_2\in\operatorname{Irr}_{\operatorname{gen}}\left(G_n(F\times F)\right)$, we define $c(\pi)$ to be $c(\pi_1)+c(\pi_2)$.   
\end{remark}
   
 In \cite{py21}, Humphries and Jo establish that there exist explicit choices of Whittaker functions and Schwartz functions for which the Rankin-Selberg (resp. Flicker) integral is a polynomial multiple of the Rankin-Selberg (resp. Asai) $L$-function.

 To a representation $\rho\in\operatorname{Irr}_{\operatorname{gen}}\left(G_n(F)\right)$, one can attach a spherical representation $\rho_{ur}$ of $G_n(F)$ such that $L(s,\rho)=L(s,\rho_{ur})$ (Proposition 4.2 of \cite{py21}). Let $\pi\in\operatorname{Irr}_{\operatorname{gen}}\left(G_n(E)\right)$. We associate a meromorphic function $p_{\pi}(s)$ to $\pi$ as follows.

$$p_{\pi}(s)\coloneqq\begin{dcases}\frac{L(s,\pi_{ur},\mathrm{As})}{L(s,\pi,\mathrm{As})} & \text { if } E/F \text{ is a quadratic extension}. \\ 
       \frac{L(s,{\pi_{1}}_{ur}\times{\pi_{2}}_{ur})}{L(s,\pi,\mathrm{As})} & \text { if } E=F\times F,\>\pi=\pi_1\otimes\pi_2.\end{dcases}.$$

By Proposition $5.8$ (split case) and Proposition $6.12$ (inert case) in \cite{py21}, $p_{\pi}(s)$ is in fact a polynomial in $s$.

    \begin{thm}\label{5.1}
        For every irreducible generic representation $\pi$ of $G_n(E)$, there exists $W_0\in\mathcal{W}(\pi,\psi_n)$ and $\Phi_0\in\mathcal{S}^{c(\pi)}$ such that $$\frac{I(s,W_0,\Phi_0)}{L(s, \pi, \mathrm{As})}=\begin{cases}p_{\pi}(s) & \text { if } E/F \text{ is a quadratic extension}. \\ 
       \delta_{n}(a(\beta))|\operatorname{det}(a(\beta))|^{-s}p_{\pi}(s) & \text { if } E=F\times F,\>\tau=(\beta,-\beta).\end{cases},$$
        where $a(\beta)$ is the diagonal matrix $\operatorname{diag}(\beta^{n-1},...,\beta,1)\in G_n(F)$.          
     \end{thm}

    \begin{proof}
        The result follows in the case when $E/F$ is a quadratic extension by Theorem $6.10$ and Proposition $6.12$ of \cite{py21} and in the case when $E=F\times F$ and $\beta=1$ by Theorem $5.6$ and Proposition $5.8$ of \cite{py21}. 

        Note that for $\psi_0(x,y):=\psi^{\prime}(x)\psi^{\prime}(-y)$ and $W_0\in\mathcal{W}(\pi,\psi_0)$ corresponding to $\beta=1$ case, we must have that $W_0^{\prime}$ defined as
        $$ {W_0}^{\prime}(g)=W_0(a(\beta)g),\>g\in G_n(E) $$
        lies in $\mathcal{W}(\pi,\psi_n)$, where $\psi(x,y)=\psi^{\prime}(\beta x)\psi^{\prime}(-\beta y)$. 
        
        Then by the change of variable $h\mapsto a(\beta)^{-1}h$,
        \[\frac{I(s,{W_0}^{\prime},\Phi_0)}{L(s, \pi, \mathrm{As})}=\delta_{n}(a(\beta))|\operatorname{det}(a(\beta))|^{-s}\frac{I(s,{W_0},\Phi_0)}{L(s, \pi, \mathrm{As})}. \]
     This completes the proof.   
    \end{proof}

Let $\pi$ be an irreducible generic representation of $G_n(E)$, with the assumption that it is nearly tempered in the inert case. The following results are deduced as a consequence of Theorem \ref{5.1}.

\begin{thm}\label{5.2}
     Suppose that the family $\mathcal{J}(\pi)$ has an exceptional pole of level $m$ at some complex number $s_0$ and $c(\pi)>m$. Then $p_{\pi}(s_0)=0$.
\end{thm}

\begin{proof}
    Pick $W_0\in\mathcal{W}(\pi,\psi_n)$ and $\Phi_0\in\mathcal{S}^{c(\pi)}$ as in Theorem \ref{5.1}. Assume to the contrary that $p_{\pi}(s_0)\neq 0$. 

As $c(\pi)>m$, $\Phi_0\in\mathcal{S}^{m+1}$, so the quotient $I(s_0,W,\Phi)/L(s_0, \pi, \mathrm{As})$ is not identically zero for $\Phi\in\mathcal{S}^{m+1}$ as ${I(s_0,{W_0},\Phi_0)}/{L(s_0, \pi, \mathrm{As})}\neq 0$. Let $d$ be the maximal order of the pole $s=s_0$ of the family $\mathcal{J}(\pi)$. Then,
\[B_{s_0}(W,\Phi)=\left(\lim_{s\to s_0} (s-s_0)^d L(s,\pi,\mathrm{As})\right)\frac{I(s_0,W,\Phi)}{L(s_0, \pi, \mathrm{As})}\]
for all $W\in\mathcal{W}(\pi,\psi_n),\>\Phi\in\mathcal{S}$. By Theorem \ref{3.1}(3), the limit in the  above equation is non-zero. Hence, the bilinear form $B_{s_0}$ does not vanish identically for $\Phi\in\mathcal{S}^{m+1}$, leading to a contradiction.
\end{proof}

Analogous to the non-archimedean theory of the conductor exponent, the archimedean conductor exponent $c(\pi)$ is zero if and only if $\pi$ is unramified. Therefore, we obtain:

\begin{corollary}\label{5.3}
   Suppose that $\pi$ is ramified and the family $\mathcal{J}(\pi)$ has an exceptional pole of level $0$ at some complex number $s_0$. Then $p_{\pi}(s_0)=0$. 
\end{corollary}

As a non-zero complex polynomial has a finite number of roots, the following result must hold.

\begin{corollary}

    The family $\mathcal{J}(\pi)$ can have at most finitely many exceptional poles of level $m$ for $m<c(\pi)$. In particular,  if $\pi$ is ramified, the family $\mathcal{J}(\pi)$ can have at most finitely many exceptional poles of level $0$.
\end{corollary}

In the inert case, we expect Theorems \ref{1.1} and \ref{5.2} to hold true, not only for nearly tempered representations but also for any irreducible generic representation. Nevertheless, it is crucial to establish the non-vanishing in Theorem \ref{3.1}$(3)$ for the broader category of irreducible generic representations to support this claim.

\section*{Acknowledgements}  I would like to thank my advisor Professor Ravi Raghunathan for his constant encouragement, invaluable comments, and for carefully reading several preliminary versions of this paper. I also wish to thank Professor Dipendra Prasad and Professor U.K. Anandavardhanan for several useful remarks.



\end{document}